\documentclass[11pt]{article}
\usepackage{amsmath}
\usepackage{amssymb}
\usepackage{amsthm}
\usepackage{graphicx}
\usepackage{caption}
\usepackage{subcaption}
\usepackage{blkarray}
\usepackage{enumerate}
\usepackage{enumitem}
\usepackage[font=small,skip=0pt]{caption}
\usepackage{xcolor}
\usepackage{verbatim}
\usepackage{geometry}
\usepackage{wrapfig}
\usepackage{color,soul}
\usepackage{float}
\usepackage[stable]{footmisc}
\usepackage{blindtext}
\usepackage{calc}
\usepackage{upgreek}
\usepackage{amsthm}
\usepackage[font=small,skip=0pt]{caption}
\usepackage[colorinlistoftodos]{todonotes}
\usepackage{url}
\usepackage{multirow}
\usepackage{stmaryrd}
\usepackage{trimclip}

\makeatletter
\DeclareRobustCommand{\shortto}{%
	\mathrel{\mathpalette\short@to\relax}%
}

\newcommand{\short@to}[2]{%
	\mkern2mu
	\clipbox{{.3\width} 0 0 0}{$\m@th#1\vphantom{+}{\shortrightarrow}$}%
}

\newcommand{\short@ta}[2]{%
	\mkern2mu
	\clipbox{{.3\width} 0 0 0}{$\m@th#1\vphantom{+}{\shortleftarrow}$}%
}
\makeatother

\numberwithin{equation}{section}
\newtheorem{Theorem}{Theorem}
\newtheorem{Corollary}{Corollary}

\newtheorem*{Note*}{Note}
\newtheorem{Proposition}{Proposition}
\newtheorem{Remark}{Remark}

\newtheorem*{Recall*}{Recall}
\graphicspath{{Figures/}}
\usepackage{color}

\definecolor{ao(english)}{rgb}{0.0, 0.5, 0.0}

\usepackage[utf8]{inputenc}
\usepackage[titles]{tocloft}

\title{Finite-Time Ruin for the Compound Markov Binomial Risk Model}
\author{
        Zbigniew Palmowski$^{a,}$\footnote{First author e-mail address: zbigniew.palmowski@gmail.com}\,, \,  Lewis Ramsden$^{b,}$\footnote{Corresponding author e-mail address: lewis.ramsden@york.ac.uk}\, and   Apostolos D. Papaioannou$^{c,}$\footnote{Third author  e-mail address: papaion@liverpool.ac.uk}     \\
      \\ $^a$Department of Applied Mathematics\\
       Wroc\l{}aw University of Science and Technology \\
       Wroc\l{}aw, Poland; \\
        $^b$ School for Business and Society, University of York\\
       York, Yorkshire, YO10 5DD, United Kingdom;\\
       $^c$Institute for Financial and Actuarial Mathematics \\
               Department of Mathematical Sciences\\
        University of Liverpool, L69 7ZL,  United Kingdom
       }
 \date{}

\begin{document}
\maketitle

\begin{abstract}
In this paper, we study finite-time ruin probabilities for the compound Markov binomial risk model - a discrete-time model where claim sizes are modulated by a finite-state ergodic Markov chain. In the classic (non-modulated) case, the risk process has interchangeable increments and consequently, its finite-time ruin probability can be obtained in terms of Tak\'acs' famous Ballot Theorem results. Unfortunately, due to the dependency of our process on the state(s) of the modulating chain these do not necessarily extend to the modulated setting. We show that a general form of the Ballot Theorem remains valid under the stationary distribution of the modulating chain, yielding a Takács-type expression for the finite-time ruin probability which holds only when the initial surplus is equal to zero. For the case of arbitrary initial surplus, we develop an approach based on multivariate Lagrangian inversion, from which we derive distributional results for various hitting times of the risk process, including a Seal-type formula for the finite-time ruin probability.
\end{abstract}


\section{Introduction}

The compound binomial risk model is a discrete-time analogue of the classic Cramér - Lundberg model, for which the insurer receives a unit premium per period and claims may occur independently at each time step according to a common distribution. This discrete model offers analytical tractability and natural applicability to insurance operations evolving over fixed-time periods. Among many key quantities of interest, the finite and infinite-time ruin probabilities defined as the probability that the insurer’s surplus becomes negative within a fixed or infinite number of periods, have been considered by several authors within the literature (see for example  \cite{avram2019first, cossette2004exact, dickson1994some, gerber1988mathematical, lefevre2008finite,  li2013finite, shiu1989probability, willmot1993ruin} and references therein).  

An interesting observation that has motivated different approaches to solving the ruin problem is its intimate connection with another famous problem within applied probability, known as the ballot problem which concerns the polling results of two candidates, A and B. This problem was first considered by Takács \cite{Takacs1962} and its solution, which along with that of a more general problem are known collectively as Ballot Theorems (see \cite{Takacs1967-rp}). It can be shown that the more general ballot problem is equivalent to that of the finite-time ruin problem of the compound binomial model with zero initial surplus and thus, emits the same expression. Moreover, by considering the time-reversed compound binomial process, which can be shown to exhibit the same law as the corresponding forward process, the Ballot Theorem can also be used to derive the so-called Seal-type formula for the finite-time ruin probability with arbitrary initial surplus (see \cite{lefevre2008finite} for details). 

This paper investigates an extension of the classic compound binomial model by allowing claim sizes to be modulated by an external, time-homogeneous, ergodic Markov chain. The resulting process - commonly known as the compound Markov binomial model  - allows the claim distribution to vary over time according to the state of an external environment. Such models have been proposed to capture the influence of economic regimes, seasonal cycles, or other systemic factors on the insurance portfolio (see for example \cite{cossette2003ruin, cossette2004exact, yuen2006some}). The surplus process in this setting becomes a Markov additive process: a natural framework for modeling stochastic systems with regime switching (see \cite{palmowski2024exit, palmowski2024gerber}).

Unlike the classic case, the modulated surplus process no longer has cyclically stationary increments and the classical Ballot Theorem of Takács no longer applies. Nevertheless, we show that a more general form of the Ballot Theorem - originally due to Kallenberg \cite{kallenberg1999ballot} - remains valid under the assumption of a stationary Markov chain. Consequently, we derive a Takács-type expression for the finite-time ruin (survival) probability for the compound Markov binomial model with zero initial surplus initial in the stationary case. For general initial surplus, the situation is more complex. Unlike the classic model, time-reversal of the modulated surplus process leads to a different joint distribution, and thus the survival probability depends not only on the path of the surplus but also on the distribution of the modulating state at the time of ruin. This lack of symmetry means that the methods of \cite{lefevre2008finite} no longer hold and an alternative approach must be considered.

It is well known that Lundberg's fundamental equation and, more specifically its (non-trivial) solution, play a vital role in risk and ruin theory. In particular, \cite{li2013finite} shows that the form of Lundberg's equation is that of a Lagrangian transformation and consequently that the distribution of the upper hitting time belongs to the Lagrangian family of distributions. Moreover, by applying Lagrangian inversion methods, \cite{li2013finite} are able to recover the Seal-type formula for the finite-time ruin probability (see also \cite{avram2019first}). Based on the results of \cite{palmowski2024exit} (see also \cite{palmowski2024gerber}), we will generalise this approach and show that Lundberg's fundamental equation for the compound Markov binomial process is of the form of a multivariate Lagrangian transformation and perform multivariate Lagrangian inversion techniques to derive several distributional results relating to the modulated process, including a Seal's-type formula for the finite-time ruin probability, expressed as an explicit sum over weighted paths, incorporating the state of the environment. These methods provide a richer structural understanding of ruin events in Markov additive settings and open avenues for further analysis and numerical implementation.

The structure of the paper is as follows: In Section 2, we introduce the compound Markov binomial risk model, its connection to discrete Markov Additive Processes (Markov Additive Chains) and some important results. In Section 3, we show that the general Ballot Theorem of \cite{kallenberg1999ballot} holds for Markov-modulated risk processes in the stationary case. In Section 4, we consider the case of zero initial surplus and show the connection between the finite-time ruin probability and the Ballot Theorem result. Moreover, we develop the multivariate Lagrangian framework and derive distributional results of various hitting times, including the time to ruin. Finally, in Section 5, we show how the results of the multivariate Lagrangian inversions can be used to derive a Seal-type expression for the finite-time ruin probability with arbitrary initial surplus.

\section{Preliminaries}

Consider a discrete-time, bivariate risk process $(X,J)=\{\left(X_n,J_n\right)\}_{n \in \mathbb{N}}$, where $\{J_n\}_{n \in \mathbb{N}}$ is an ergodic Markov chain with state space $E = \{1, 2, \ldots, N\}$, probability transition matrix $\bold{P}$ and stationary distribution given by the $N$ dimensional (row) vector $\boldsymbol{\pi}$, and
\begin{equation} \label{eq:RW}
	X_n = x + n - S_n, \qquad \text{with} \qquad	S_n = \sum^{n}_{k=1}C_k,
\end{equation}
where $\{C_k\}_{k \in \mathbb{N}^+}$ is a sequence of conditionally i.i.d.\,non-negative, integer valued random variables with common (conditional) distribution $\lambda_{ij}(m):=\mathbb{P}(C_k=m, J_k=j|J_{k-1}=i)$ and $S_0=0$. In the risk theory setting, the level process $\{X_k\}_{k \in \mathbb{N}}$ represents the surplus of an insurance company having initial surplus $x \in \mathbb{N}$, receiving unit premium income at the start of each period and liable to random, integer claim payments $\{C_k\}_{k\in \mathbb{N}^+}$, that are (possibly) dependent on the state of the Markov chain $\{J_n\}_{n \in \mathbb{N}}$, representing some external influencing environment , e.g.\,weather or market conditions. This is known in the literature as the compound Markov binomial model (see \cite{cossette2004exact}, \cite{yuen2006some} and references therein). Moreover, let $\bold{\Lambda}(m)$ denote the $N\times N$ matrix with elements $\lambda_{ij}(m)$ for all $i,j \in E$, then we define the probability generating matrix
\begin{equation}\label{eq:pgf}
\widehat{\bold{\Lambda}}(z) := \mathbb{E}\left(z^{C_1};J_1 \right) = \sum_{m=0}^\infty z^m \bold{\Lambda}(m),
\end{equation}
with $\widehat{\bold{\Lambda}}(1) = \bold{P}$, where we have used the notation $\mathbb{E}\left(\cdot\,;J_1 \right)$ to represent the $N\times N$ matrix with $i,j$-th element $\mathbb{E}(\cdot\, 1_{\{J_1=j\}}|J_0=i)$ having corresponding probability matrix $\mathbb{P}(\cdot, J_1)$ with elements $\mathbb{P}(\cdot, J_1=j | J_0 = i)$ for all $i,j \in E$.  Throughout the remainder of this paper, we will assume that all elements of $\bold{\Lambda}(0)$ are strictly positive, i.e.\,$\lambda_{ij}(0) > 0$ for all $i,j \in E$ and that $\bold{\Lambda}(0)$ is invertible. 

The dependence of the random variables $\{C_k\}_{k \in \mathbb{N}^+}$ on the Markov chain $\{J_n\}_{n\in \mathbb{N}}$ means that, unlike traditional random walks,  $\{X_{n}\}_{n \in \mathbb{N}}$ only exhibits conditionally stationary and independent increments. That is, given the event $\{J_{n-1}=i\}$, it follows that 
\begin{equation*}
	\left(X_n-X_{n-1}, J_n\right) \overset{d}{=} 	\left(X_1-X_{0}, J_1\right)
\end{equation*}
given $\{J_0=i\}$. This property is known as the Markov additive property (see \cite{kyprianou2008fluctuations}) and $(X,J)$ is referred to as a Markov Additive Chain (MAC) (see \cite{palmowski2024exit}).

Of particular interest in this paper are the so-called upper and lower hitting times for $\{X_n\}_{n\in \mathbb{N}}$ of fixed levels, say $a, b \in \mathbb{Z}$ with $a > b$, defined by 
\begin{equation*}
	\tau^+_a = \min\{n \geq 0: X_{n} \geq a \} \qquad \text{and} \qquad 	\tau^-_b = \min\{n > 0: X_{n} \leq b \}. 
\end{equation*}

\noindent For the upper hitting time, $\tau^+_a$, we define the transform $\bold{G}_v(a)$, for $v \in (0,1]$, by
\begin{equation}\label{eq:series}
	\bold{G}_v(a) := \mathbb{E}\left(v^{\tau^+_a}; J_{\tau^+_a} \right) = \sum_{n=0}^\infty v^n \mathbb{P}(\tau_a^+ = n , J_{n}),
\end{equation}
which, by the Markov additive property, satisfies 
\begin{equation} \label{eq:Ga}
	\bold{G}_v(a) = \bold{G}_v^a,
\end{equation}
where $\bold{G}_v:=\bold{G}_v(1)$. Using a simple conditioning argument on the first period of time, it follows that
\begin{equation} \label{eq:upper}
	\bold{G}_v = v \sum_{m=0}^\infty \bold{\Lambda}(m)\bold{G}_v^m
\end{equation}
and thus, $\bold{G}_v$ can be seen as the (right) solution of Lundberg's equation
\begin{equation} \label{eq:Lundberg}
	\bold{G}_v =v\widehat{\bold{\Lambda}}^{\shortleftarrow}(\bold{G}_v),
\end{equation}
where $\widehat{\bold{\Lambda}}^{\shortleftarrow}(\cdot)$ represents the transform $\widehat{\bold{\Lambda}}(\cdot)$, defined in Eq.\,\eqref{eq:pgf}, with the argument acting from the right. This is an important distinction to make since $\bold{\Lambda}(\cdot)$ and $\bold{G}_v$ are not necessarily commutative and thus  $\widehat{\bold{\Lambda}}^{\shortleftarrow}(\bold{G}_v) \neq \widehat{\bold{\Lambda}}^{\shortrightarrow}(\bold{G}_v)$ - the corresponding transform with argument acting from the left. In fact, we see that the so-called left and right solutions to Lundberg's equation yield different quantities, both of which play a fundamental role in the fluctuation theory of MACs (see below and for further details \cite{Ivanovs2019, palmowski2024exit}).

\begin{Remark}
It is worth pointing out that the matrix $\bold{G}_v$, as the (right) solution of the Lundberg's equation, can be found explicitly in a few special cases or approximated via a number of numerical algorithms (see \cite{bini2005numerical} and references therein). However, this is not the aim of this paper and instead, we merely aim to exploit the form of Eq.\,\eqref{eq:upper} to obtain the following results.
\end{Remark}

\subsection{Time reversal} \label{sec:rev}

	Let us define, for some fixed $n\in \mathbb{N}$, the time-reversed processes
\begin{equation*}
	\widetilde{X}_k := X_n-X_{n-k} \qquad \text{and} \qquad \widetilde{J}_k:=J_{n-k},
\end{equation*}
for $k = 0, 1, \ldots, n$ where $\widetilde{X}_0 = 0$. Then, assuming that $\{J_k\}_{k\in \mathbb{N}}$ has stationary initial distribution $\boldsymbol{\pi}$, $\{\widetilde{J}_{k}\}_{0\leq k\leq n}$ is again a time-homogeneous Markov chain having the same initial stationary distribution and $\{\widetilde{X}_k\}_{0 \leq k \leq n}$ is also a compound Markov binomial process of the form 
\begin{equation*}
	\widetilde{X}_k = k - \widetilde{S}_k, \qquad \text{with} \qquad \widetilde{S}_k = \sum_{i=1}^k \widetilde{C}_i,
\end{equation*}
with conditional jump size distribution $\widetilde{\lambda}_{ij}(m) :=	 \frac{\pi_j}{\pi_i}\lambda_{ji}(m)$. In particular, if we define $\Delta_{\boldsymbol{\pi}}:=diag(\boldsymbol{\pi)}$, then the corresponding jump distribution matrix, denoted $\widetilde{\bold{\Lambda}}(m)$, satisfies 
\begin{equation} \label{eq:TRP}
	\widetilde{\bold{\Lambda}}(m) = \Delta_{\boldsymbol{\pi}}^{-1}\bold{\Lambda}^\top(m)\Delta_{\boldsymbol{\pi}}.
\end{equation}

\noindent Moreover, if we define
\begin{equation} \label{eq:Grev}
	\widetilde{\bold{G}}_v^a = \mathbb{E}\left(v^{\widetilde{\tau}^+_a}; \widetilde{J}_{\widetilde{\tau}^+_a}\right) = \sum_{n=0}^\infty v^n \mathbb{P}\left(\widetilde{\tau}^+_a=n, \widetilde{J}_{\widetilde{\tau}^+_a}\right),
\end{equation}
with $\widetilde{\tau}^+_a := \min\{k\in \mathbb{N}: \widetilde{X}_k \geq a \in \mathbb{N}\}$, to be the time-reversed counterpart to $\bold{G}_v$, then it follows that 
\begin{equation} \label{eqn:R}
	\bold{R}_v := \Delta_{\boldsymbol{\pi}}^{-1}\widetilde{\bold{G}}_v^\top \Delta_{\boldsymbol{\pi}}
\end{equation}
is the corresponding (left) solution of Lundberg's equation such that $\bold{R}_v = v \widehat{\bold{\Lambda}}^{\shortrightarrow}(\bold{R}_v)$, where $\widehat{\bold{\Lambda}}^{\shortrightarrow}(\cdot)$ represents the transform $\widehat{\bold{\Lambda}}(\cdot)$ with argument acting from the left (see \cite{palmowski2024exit} for details).

\begin{Remark}
The existence of the two (possibly) different matrix solutions, $\bold{G}_v$ and $\bold{R}_v$, occurs since the time-reversed process is not necessarily equal in distribution to the original process. This is due to the state dependence on the Markov chain $\{J_n\}_{n \in \mathbb{N}}$. However, in the scalar case ($N=1$), it can be shown that both processes are in fact equal in distribution (see \cite{Takacs1962} among others), in which case these two solutions coincide and are equal to the smallest solution $z \in (0,1]$ to Lundberg's fundamental equation $z=v \widehat{\bold{\Lambda}}(z)$ (with $N=1$). 
\end{Remark}

\section{Ballot-Type Theorem}

{It is well-known that discrete risk models have connections to a variety of other models and applications in applied probability, e.g., random walks, queuing theory, population dynamics etc. This is one of the main reasons for such a rich literature on the topic and continued interest in the area. One very well known model which is intimately connected is the so-called \textit{ballot problem}. 
	
The ballot problem was first considered in \cite{bertrand1887solution} and describes the probability that candidate A will always be ahead of candidate B in the polls, given a fixed number of votes, and results in an elegantly simple expression. This simple model was later generalised by Tak\'acs (\cite{Takacs1962} and \cite{Takacs1967-rp}) to any stochastic process with cyclically stationary (interchangeable) increments, a property which the classic compound binomial risk model ($N=1$) possesses. Consequently, these classic ballot-type results have seen various applications in the risk theory literature and have been used to derive a number of results concerning the so-called finite and infinite-time ruin probabilities [see \cite{lefevre2008finite} and references therein]. Unfortunately, for the general compound Markov binomial model defined in Eq.\,\eqref{eq:RW}, the dependence of the random variables $\{C_k\}_{k \in \mathbb{N}^+}$ on the external Markov chain $\{J_n\}_{n\in \mathbb{N}}$ means that the jump process $\{S_{n}\}_{n \in \mathbb{N}}$ does not exhibit cyclically stationary increments and thus, the classic ballot result of \cite{Takacs1962}, does not hold. However, Kallenberg \cite{kallenberg1999ballot} later proved that the cyclicity assumption is not necessary and that stationary increments alone is sufficient.

In the following proposition, we show that if the Markov chain $\{J_n\}_{n\in \mathbb{N}}$ has stationary initial distribution ($\boldsymbol{\pi}$), then the increments form a so-called stationary sequence, a consequence of which is that $\{S_n\}_{n \in \mathbb{N}}$ also has stationary increments.  In the following we use the notation $\mathbb{P}_{\boldsymbol{\pi}}$ to represent the distribution under the stationary initial distribution.

\begin{Proposition} \label{Prop:Increm}
	Assume that the Markov chain $\{J_n\}_{n\in \mathbb{N}}$ has stationary initial distribution $\boldsymbol{\pi}$, then $\{C_k\}_{k \in \mathbb{N}^+}$ forms a stationary sequence. That is, for every $1 \leq m \leq n$, $r \leq n-m$ and integers $0 = i_0 < i_1 < i_2 < \ldots < i_r \leq n-m$, we have  
	\begin{equation*}
		\mathbb{P}_{\boldsymbol{\pi}}\left(C_{1+i_0}, C_{1+i_1}, \ldots, C_{1+i_r}\right)=	\mathbb{P}_{\boldsymbol{\pi}}\left(C_{m+i_0}, C_{m+i_1}, \ldots, C_{m+i_r}\right) . 
	\end{equation*}
\end{Proposition}
\begin{proof}
	See Appendix. 
\end{proof} 

\noindent Given that $\{S_n\}_{n \in \mathbb{N}}$ has stationary increments, we can apply similar arguments to those of Kallenberg \cite{kallenberg1999ballot} to derive a Ballot Theorem type result for the compound Markov binomial risk model, given by the following theorem. It is worth noting that although this paper is only concerned with integer jump sizes, for the sake of completeness, we prove the following theorem holds for general jumps (discrete and continuous).

\begin{Theorem}\label{PropBallottheorem}
	Assume that the Markov chain $\{J_n\}_{n \in E}$ has stationary initial distribution $\boldsymbol{\pi}$. Then, for $S_n = C_1 + \cdots + C_n$, with $C_k \in \mathbb{R}^+\cup \{0\}$, we have
	\begin{equation}\label{ballotidentity}\mathbb{P}_{\boldsymbol{\pi}}\left(\max_{0<k\leq n}(
		S_k-k) < 0 \big| S_n\right)=
		1-\frac{S_n}{n},\end{equation}
	for $S_n < n$ and $0$ otherwise. 
\end{Theorem}
\begin{proof}
	Let us first prove the result for jumps which are supported on the non-negative lattice $\mathbb{N}d$ for some fixed grid $d>0$. That is, for $\{S^d_n\}_{n \in \mathbb{N}}$ which represents $\{S_n\}_{n\in \mathbb{N}}$ with $C_k\in \mathbb{N}d$ for all $k \in \mathbb{N}^+$. Moreover, let us define $S^d_t:=S^d_k$ for $t\in [k, k+1)$. The process $\{S^d_t\}_{t \geq 0}$ is clearly a non-decreasing step function with values in $\mathbb{N}d$ and, from Proposition \ref{Prop:Increm}, possesses stationary increments. As such, the corresponding probability measure defining the law of $\{S^d_t\}_{t \geq 0}$, is a.s.\,singular and stationary and thus,  by Theorem 2.1 of \cite{kallenberg1999ballot}, we have
	\begin{eqnarray*}
			\mathbb{P}_{\boldsymbol{\pi}}\left(\max_{0<k\leq n}(S^d_k-k) < 0 \big|  S^d_n\right)&=&
			\mathbb{P}_{\boldsymbol{\pi}}\left(\sup_{0 \leq t\leq n}(S^d_t-t) = 0 \big| S^d_n\right) \\
			&=&
			\mathbb{P}_{\boldsymbol{\pi}}\left(\sup_{0 < t\leq n}\left(\frac{S^d_t}{t}\right) \leq 1 \big| S^d_n\right) \\
			&=&
			\mathbb{P}_{\boldsymbol{\pi}}\left(\frac{S^d_n}{nU} \leq 1\right)=
			1-\frac{S^d_n}{n},
		\end{eqnarray*}
		when $S^d_n <  n$ ($0$ otherwise), where $U$ denotes an independent standard uniform random variable, i.e. $U \sim U[0,1]$. This proves the result in the lattice case with $d >0$.  
		
		To prove the case for general jump sizes, we will consider an approximation by taking the limit as the grid size $d\downarrow 0$. First recall that any distribution function can be decomposed into its absolutely continuous and singular parts, it then suffices to consider the limit for the absolutely continuous distribution function of generic jumps $C_k$, namely $F$ which depends on the Markov chain $\{J_n\}_{n\in \mathbb{N}}$, only. Now, note that if the support of a general jump distribution, $F$, is in $[-K,K]$ for some $K$ (for all states of the Markov chain $\{J_n\}_{n \in \mathbb{N}})$, then by a step approximation of the distribution function $F$ (Heine–Cantor theorem) it follows that one can find $F^d$ such that $\lim_{d\downarrow 0} F^d(x)=F(x)$ for all $x\in [-K,K]$. Thus, if we define $F^d$ to be the distribution function of jumps supported on the lattice with grid size $d>0$, it follows that the distribution function of $\{S^d_n\}_{n \in \mathbb{N}}$ converges weakly to the general distribution function of $\{S_n\}_{n \in \mathbb{N}}$. Hence, by cutting the jumps to $C_k\wedge K$ and using the dominated convergence theorem as $K\shortrightarrow +\infty$, we obtain the convergence for general distribution functions. In turn, this means that by the Markov property of both processes, $\{S^d_n\}_{n \in \mathbb{N}}$ converges weakly to $\{S_n\}_{n \in \mathbb{N}}$ as $d\downarrow 0$. Moreover, since the supremum on a finite interval is a continuous functional, then by the so-called continuous mapping theorem we can also conclude that $\max_{k\leq n}(S^d_k-k)$ converges to $\max_{k\leq n}(S_k-k)$ as $d\downarrow 0$. As such, both sides of \eqref{ballotidentity} convergence and the statement holds for a general random walk with general jump sizes $C_k \in \mathbb{R}^+\cup \{0\}$.
	\end{proof}

\noindent Although the above result is of interest in its own right, we are primarily concerned with its application in risk theory. The rest of this paper will be concerned with the time to ruin for the general compound Markov binomial model, i.e. $\tau^-_0$. 

\section{Finite-Time Ruin - $x = 0$}

In this section, we consider the distribution of the time to ruin ($\tau^-_0$) in the case of zero initial surplus, i.e. $x=0$. Under this assumption, we can employ the Ballot-type result of Theorem \ref{PropBallottheorem} to determine the finite-time ruin probability in the form of a Tak\`acs-type formula (see \cite{Takacs1962}). However, unlike the unmodulated case, it is not possible to employ this result for the case of general initial surplus ($x \geqslant 0$) (see for example \cite{lefevre2008finite}). As such, in this section we also derive a result for the distribution of the time to ruin based on a multivariate generalisation of the Lagrange inversion considered in \cite{li2013finite} (see also \cite{avram2019first}), which can be employed to derive an expression for the finite-time ruin probability with general initial surplus.

\begin{Proposition}[Tak\'acs-Type Formula] \label{Prop:Ruin0} Assume that $\{J_n\}_{n\in \mathbb{N}}$ has stationary initial distribution $\boldsymbol{\pi}$. Then, it follows that $\mathbb{P}_{\boldsymbol{\pi}}(\tau^-_0 = 1) = \mathbb{P}_{\boldsymbol{\pi}}(C_1 \geq 1)$ and for $n \in \mathbb{N}^+$, the finite-time survival probability is given by  
	\begin{equation} \label{eq:ruin0} 
		\mathbb{P}_{\boldsymbol{\pi}}(\tau_0^- \geq n+1) = \frac{1}{n}\sum_{m=0}^{n-1} (n-m) \boldsymbol{\pi}\bold{\Lambda}^{*n}(m)\vec{\boldsymbol{e}},
	\end{equation}
	where $\vec{\boldsymbol{e}}$ is an $N$-dimensional (column) vector of units and $\bold{\Lambda}^{*n}(\cdot)$ denotes the $n$-th fold convolution of the claim size distribution matrix $\bold{\Lambda}(\cdot)$.
\end{Proposition}
\begin{proof}
	The result for $\mathbb{P}_{\boldsymbol{\pi}}(\tau^-_0 = 1)$ is trivial after recalling the unit premium income per period. In this case, ruin occurs if a non-zero claim arrives during this first period. On the other hand, for $n \in \mathbb{N}^+$ we note that for $x=0$, the risk process $X_k = k-S_k$ and thus, by considering the reflected process (at zero), i.e.\,$\{-X_k\}_{k\in \mathbb{N}}$, it is straightforward to see that 
	\begin{eqnarray*}
		\mathbb{P}_{\boldsymbol{\pi}}(\tau_0^- \geq n+1) 
		&=& \sum_{m=0}^{n-1} \mathbb{P}_{\boldsymbol{\pi}}\left(\max_{0<k\leq n}\left(S_k-k\right)<0, S_n = m\right) \\
		&=& \sum_{m=0}^{n-1} \mathbb{P}_{\boldsymbol{\pi}}\left(\max_{0<k\leq n}\left(S_k-k\right)<0 \big| S_n = m\right) \mathbb{P}_{\boldsymbol{\pi}}(S_n=m) \\
		&=& \sum_{m=0}^{n-1} \frac{n-m}{n} \,\mathbb{P}_{\boldsymbol{\pi}}(S_n=m),
	\end{eqnarray*}
	where the last equality follows from Theorem \ref{PropBallottheorem}. The result follows after noting that 
	\begin{eqnarray*}
		\mathbb{P}_{\boldsymbol{\pi}}(S_n=m) &=& \sum_{i=1}^N \sum_{j=1}^N \pi_i \mathbb{P}(S_n=m, J_n=j \big| J_0=i) \\
		&=& \boldsymbol{\pi}\bold{\Lambda}^{*n}(m)\vec{\boldsymbol{e}},
	\end{eqnarray*}
	where $\vec{\boldsymbol{e}}$ is an $N$ dimensional column vector of units and $\bold{\Lambda}^{*n}(\cdot)$ denotes the $n$-th fold convolution of $\bold{\Lambda}(\cdot)$.
\end{proof}

\begin{Remark}
	Note that when $N=1$, Proposition \ref{Prop:Ruin0} reduces to the Tak\'acs-type formula for the compound binomial risk process given, for example, by Proposition 2.2 of \cite{lefevre2008finite}.
\end{Remark}

\subsection{Multivariate Lagrangian inversion}

It is well known that Lundberg's fundamental equation plays a crucial role in risk theory, with its solution(s) being directly connected to the distribution of various hitting times (see \cite{asmussen2010ruin} and references therein). In the discrete setting, \cite{li2013finite} shows that for the compound binomial risk model, the form of Lundberg's equation identifies the upper hitting time $\tau^+_a$ as belonging to the Lagrangian family of distributions (see \cite{consul2006lagrangian}). Moreover, by applying the standard Lagrangian inversion techniques (see \cite{consul2006lagrangian} for details) and coefficient matching, the authors also derive an explicit expression for the distribution of the time to ruin. In the following, we extend these results to show that the form of Eq.\,\eqref{eq:Lundberg} identifies $\tau^+_a$, for the compound Markov binomial model, as belonging to the so-called multivariate Lagrangian family of distributions and employ inversion techniques to determine the distribution for a number of hitting times, including the time to ruin.

For convenience in the following, let us define $G_{ij}:=\left(\bold{G}_v\right)_{ij}$, $\widehat{\lambda}^{\shortleftarrow(\shortrightarrow)}_{ij}(\cdot):= \left(\widehat{\bold{\Lambda}}^{\shortleftarrow(\shortrightarrow)}(\cdot)\right)_{ij}$ and the differential operators $D_{ij} := \frac{\partial}{\partial G_{ij}}$ for all $i,j \in E$, where it is understood that these operators act on everything to their right hand side. The remaining results of this paper are direct consequences of the following proposition. 

\begin{Proposition} \label{prop:Lag}
	For some meromorphic function $f : \mathbb{R}^{N\times N} \shortrightarrow [0, \infty)$, we have
	\begin{equation} \label{eg:General}
		f\left(\bold{G}_v\right) = \sum_{n=0}^\infty v^n\rho(n), \qquad \text{with} \qquad \rho(n) = \sum_{\substack{x_{11}, \ldots, x_{NN} \geq 0; \\ x_{11} + \cdots + x_{NN} = n}} \frac{b(\vec{x}\,)}{x_{11}!\cdots x_{NN}!},
	\end{equation}
	where
	\begin{equation} \label{eq:b}
		b(\vec{x}\,) = \left[\prod_{i=1}^N \prod_{j=1}^N D_{ij}^{x_{ij}}\left(f\left(\bold{G}_v\right)\prod_{i=1}^N\prod_{j=1}^N \widehat{\lambda}^\shortleftarrow_{ij}\left(\bold{G}_v\right)^{x_{ij}} \left|\bold{I} - \bold{\Gamma}^\shortleftarrow 
		\left(\bold{G}_v\right)\right|\right)\right]_{\bold{G}_v = \bold{0}},
	\end{equation}
	with $\bold{I}$ denoting the $N^2\times N^2$ identity matrix and $\bold{\Gamma}^\shortleftarrow\left(\bold{G}_v\right)$ an $N^2 \times N^2$ matrix with elements $\gamma_{ij,kl}\left(\bold{G}_v\right) = G_{ij}D_{kl}\left\{ \ln\left[\widehat{\lambda}^\shortleftarrow_{ij}\left(\bold{G}_v\right)\right]\right\}$, for $i,j,k,l \in E$. 
\end{Proposition}

\begin{proof}
	Let us first note that by Eq.\,\eqref{eq:upper}, we have a system of $N\times N$ equations
	\begin{equation*}
		G_{ij} = v \widehat{\lambda}^\shortleftarrow_{ij}\left(\bold{G}_v\right), \qquad i,j \in E,
	\end{equation*}
	which are in the form of so-called Lagrangian transformations with $\widehat{\lambda}^\shortleftarrow_{ij}(\bold{0}) = \left(\bold{\Lambda}(0)\right)_{ij}=\lambda_{ij}(0)$ for $i,j \in E$. Thus, given that $\lambda_{ij}(0) \neq 0$ for all $i,j \in E$, for a meromorphic function $f : \mathbb{R}^{N \times N} \shortrightarrow [0,\infty)$ the multivariate Lagrangian inversion formula (see \cite{consul2006lagrangian} for details) yields
	\begin{equation*}
		f\left(\bold{G}_v\right) = \sum_{x_{11}=0}^\infty \cdots \sum_{x_{NN}=0}^\infty \frac{v^{x_{11}+\cdots + x_{NN}}}{x_{11}! \cdots x_{NN}!}b(\vec{x}\,), 
	\end{equation*}
	where $b(\vec{x}\,)$ is defined in Eq.\,\eqref{eq:b}. The result follows by collecting common powers of $v$ and re-ordering the summations.
\end{proof}

\begin{Corollary} \label{Cor:1}
	The matrix $\bold{G}^a_v$ has the multivariate Lagrange expansion 
	\begin{equation} \label{eq:G}
		\bold{G}^a_v = \sum_{n=0}^\infty v^n \bold{Q}(n, a), \qquad \text{with} \qquad \bold{Q}(n, a) = \sum_{\substack{x_{11}, \ldots, x_{NN} \geq 0; \\ x_{11} + \cdots + x_{NN} = n}} \frac{1}{x_{11}!\cdots x_{NN}!}\bold{B}(\vec{x}, a),
	\end{equation}
	and
	\begin{equation} \label{eq:B}
		\bold{B}(\vec{x}, a) =  \left[\prod_{i=1}^N \prod_{j=1}^N D_{ij}^{x_{ij}}\left(\bold{G}^a_v  \prod_{i=1}^N\prod_{j=1}^N \widehat{\lambda}^\shortleftarrow_{ij}\left(\bold{G}_v\right)^{x_{ij}} \left|\bold{I} - \bold{\Gamma}^\shortleftarrow 
		\left(\bold{G}_v\right)\right|\right)\right]_{\bold{G}_v = \bold{0}},
	\end{equation}
	where the differential operators are understood to act element-wise on $\bold{G}^a_v$.
\end{Corollary}

\begin{proof}
	The proof follows directly by considering functions $f\left(\bold{G}_v\right) = \left(\bold{G}_v^a\right)_{mn}$ for each $m,n \in E$ in Theorem \ref{Thm:Main} and combining the corresponding results in matrix form.
\end{proof}

\noindent Comparing the above result with the infinite summation representation for $\bold{G}_v^a$ given in Eq.\,\eqref{eq:series} produces the following result for the upper hitting time of the compound Markov binomial model which generalises Kemperman's formula \cite{kemperman1961passage} (see also \cite{avram2019first} for a discussion of these results in the scalar case).

\begin{Theorem} \label{Thm:Main}
	The distribution of the upper hitting time of the level $a\in \mathbb{N}$ can be expressed, for $n \geq a$, as 
	\begin{eqnarray} \label{eq:Gtau}
		\mathbb{P}\left( \tau^+_a = n,  J_{\tau^+_a}\right)  = \bold{Q}(n,a) ,
	\end{eqnarray}
	where $\bold{Q}(n,a )$ is defined in Eq.\,\eqref{eq:G}. 
\end{Theorem}

\begin{proof}
	The result follows by matching the coefficients of $v^n$ in Eqs.\,\eqref{eq:series} and \eqref{eq:G}.
\end{proof}

\noindent The result of Theorem \ref{Thm:Main} follows from the fact that Eq.\,\eqref{eq:upper} is of the form of a Lagrangian transformation. With this in mind, it is clear that $\bold{R}_v$ also satisfies a similar result which is presented in the following Corollary. 

\begin{Corollary} \label{cor:rev}
	The matrix $\bold{R}^a_v$ has the multivariate Lagrange expansion
	\begin{equation*}
		\bold{R}^a_v = \sum_{n=0}^\infty v^n \bold{V}(n,a), \quad \text{with} \quad \bold{V}(n,a) =  \sum_{\substack{x_{11}, \ldots, x_{NN} \geq 0; \\ x_{11} + \cdots + x_{NN} = n}} \frac{1}{x_{11}!\cdots x_{NN}!}\bold{C}(\vec{x}, a),
	\end{equation*}
	and
	\begin{equation*}
		\bold{C}(\vec{x}, a) =  \left[\prod_{i=1}^N \prod_{j=1}^N D_{ij}^{x_{ij}}\left(\bold{R}^a_v  \prod_{i=1}^N\prod_{j=1}^N \widehat{\lambda}^{\shortrightarrow}_{ij}\left(\bold{R}_v\right)^{x_{ij}} \left|\bold{I} - \bold{\Gamma}^\shortrightarrow
		\left(\bold{R}_v\right)\right|\right)\right]_{\bold{R}_v = \bold{0}},
	\end{equation*}
	where $\bold{\Gamma}^{\shortrightarrow}_{ij}(\cdot)$ is defined as in Proposition \ref{prop:Lag} with $\widehat{\lambda}_{ij}^\shortleftarrow(\cdot)$ replaced with the corresponding version $\widehat{\lambda}_{ij}^\shortrightarrow(\cdot)$. Moreover, we have 
	\begin{equation} \label{eq:revhitting}
		\mathbb{P}\left(\widetilde{\tau}^+_a = n, \widetilde{J}_{\widetilde{\tau}^+_a}\right) = \Delta_{\boldsymbol{\pi}}^{-1}\bold{V}^\top(n,a)\Delta_{\boldsymbol{\pi}}.
	\end{equation}
\end{Corollary}
\begin{proof}
	The first result follows by similar reasoning to that of Corollary \ref{Cor:1}. To prove Eq.\,\eqref{eq:revhitting}, we note that from the first part of the above result and Eqs.\,\eqref{eq:Grev}-\eqref{eqn:R}, we have
	\begin{eqnarray*}
		\sum_{n=0}^\infty v^n \bold{V}^\top(n,a) = \left(\bold{R}^a_v\right)^\top &=& \Delta_{\boldsymbol{\pi}}\widetilde{\bold{G}}_v^a \Delta_{\boldsymbol{\pi}}^{-1} \\
		&=& 
		\Delta_{\boldsymbol{\pi}}\sum_{n=0}^\infty v^n 	\mathbb{P}\left(\widetilde{\tau}^+_a = n, \widetilde{J}_{\widetilde{\tau}^+_a}\right) \Delta_{\boldsymbol{\pi}} ^{-1}.
	\end{eqnarray*}
	The result follows by multiplication of $\Delta_{\boldsymbol{\pi}}^{-1}$ and $\Delta_{\boldsymbol{\pi}}$ on the left and right hand side respectively and matching coefficients.
\end{proof}

\begin{Remark}
When $N=1$, the matrices $\bold{Q}(n,a)$ and $\bold{V}(n,a)$ are equivalent and equal to
\begin{equation*}
\bold{Q}(n,a) = \bold{V}(n,a) = \frac{a}{n} \lambda^{*n}(n-a)
\end{equation*}
(see A2 of Appendix) and Theorem \ref{Thm:Main} reduces to Proposition 3.1 of \cite{li2013finite}. The reader should keep this result in mind in the next section(s) in order to see the connection between the multivariate results of this paper and their scalar counterparts. 
\end{Remark}

\subsection{Distribution of the time to ruin}

It turns out that the result(s) of Corollary \ref{Cor:1} and, more importantly Corollary \ref{cor:rev}, also allow us to determine the distribution of the time to ruin, as can be seen in the following theorem. 

\begin{Theorem}  Assume that $\{J_n\}_{n \in \mathbb{N}}$ has stationary initial distribution $\boldsymbol{\pi}$. Then, it follows that 	\begin{equation*}
		\mathbb{P}(\tau^-_0 = 1, J_1) = \bold{P} - \bold{\Lambda}(0),
	\end{equation*}
	whilst for $n > 1$, we have 
	\begin{equation}
		\mathbb{P}\left(\tau_0^- = n, J_n \right) = \sum_{x=1}^{n-1} \bold{V}(n-1,x)\bold{P} - \sum_{x=1}^{n}\bold{V}(n,x) .
	\end{equation}
\end{Theorem}
\begin{proof}
	The result for $n=1$ follows directly by conditioning on the first period of time and noting that $\sum_{m=0}^\infty \bold{\Lambda}(m) = \bold{P}$. For $n >1$, let us denote $\bold{\Phi}(x) := \mathbb{E}_{\boldsymbol{\pi},x}\left(v^{\tau^-_0}1_{\{\tau^-_0 < \infty\}};J_{\tau^-_0} \right)$. Then, since it is assumed that $\bold{\Lambda}(0)$ is invertible, by Proposition 2 of \cite{palmowski2024gerber} it follows that 
	\begin{eqnarray*}
		\bold{\Phi}(1) &=& \bold{\Lambda}(0)^{-1}\sum_{x=1}^\infty \sum_{y=0}^\infty \bold{R}^x_v \bold{\Lambda}(x+y+1) \\
		&=& \bold{\Lambda}(0)^{-1} \bold{R}_v \sum_{y=2}^\infty \left(\sum_{x=0}^{y-2} \bold{R}^x_v \right)\bold{\Lambda}(y) \\
		&=& \bold{\Lambda}(0)^{-1} \bold{R}_v \sum_{y=2}^\infty \left(\bold{I}-\bold{R}_v\right)^{-1}\left(\bold{I}-\bold{R}_v^{y-1}\right)\bold{\Lambda}(y) \\
		&=& \bold{\Lambda}(0)^{-1} \bold{R}_v \left(\bold{I}-\bold{R}_v\right)^{-1} \bold{R}_v^{-1} \left[ \bold{R}_v \sum_{y=2}^\infty \bold{\Lambda}(y) -  \sum_{y=2}^\infty \bold{R}_v^{y}\bold{\Lambda}(y) \right] \\
		&=& \bold{\Lambda}(0)^{-1} \left(\bold{I}-\bold{R}_v\right)^{-1} \left[ (\bold{I}-\bold{R}_v)\bold{\Lambda}(0) - (\widetilde{\bold{\Lambda}}^{\shortleftarrow}(\bold{R}_v) - \bold{R}_v\bold{P} ) \right] \\
		&=& \bold{I} -  \bold{\Lambda}(0)^{-1} \left(\bold{I}-\bold{R}_v\right)^{-1} \bold{R}_v(v^{-1}\bold{I}- \bold{P}),
	\end{eqnarray*}
	since $\widehat{\bold{\Lambda}}^{\shortrightarrow}(\bold{R}_v) = v^{-1}\bold{R}_v$ (see Eq.\eqref{eqn:R}). Using similar arguments to Remark 4 of \cite{palmowski2024gerber}, $\bold{R}_v$ is a sub-stochastic matrix and thus all eigenvalues have absolute value less than 1. As such, by the geometric summation formula, the above is equivalent to
	\begin{eqnarray*}
		\bold{\Phi}(1) = \bold{I} -  \bold{\Lambda}(0)^{-1} \sum_{x=1}^\infty \bold{R}_v^x(v^{-1}\bold{I}- \bold{P}). 
	\end{eqnarray*}
	Now, by substituting the result of Corollary \ref{cor:rev} into the above, we have 
	\begin{eqnarray*}
		\bold{\Phi}(1) &=& \bold{I} -  \bold{\Lambda}(0)^{-1} \sum_{x=1}^\infty \sum_{n=x}^\infty v^n \bold{V}(n,x)(v^{-1}\bold{I}- \bold{P}) \notag \\
		&=& \bold{I} -  \bold{\Lambda}(0)^{-1} \sum_{n=1}^\infty v^{n-1} \sum_{x=1}^n  \bold{V}(n,x) +  \bold{\Lambda}(0)^{-1} \sum_{n=1}^\infty v^{n} \sum_{x=1}^n  \bold{V}(n,x) \bold{P} \notag \\
		&=& \sum_{n=1}^\infty v^n \bold{\Lambda}(0)^{-1}  \left(\sum_{x=1}^n \bold{V}(n,x)\bold{P} - \sum_{x=1}^{n+1}\bold{V}(n+1,x) \right), 
	\end{eqnarray*}
	since $\bold{V}(1,1)=\Delta_{\boldsymbol{\pi}}^{-1}\widetilde{\bold{\Lambda}}(0)^\top \Delta_{\boldsymbol{\pi}} = \bold{\Lambda}(0)$. Then, by noting $\bold{\Phi}(1) = \sum_{n=1}^\infty v^n \mathbb{P}_{1}(\tau^-_0 = n, J_n)$ and matching coefficients, it follows that 
	\begin{equation} \label{eq:dist1}
		\mathbb{P}_{1}(\tau^-_0 = n, J_n) = \bold{\Lambda}(0)^{-1}  \left(\sum_{x=1}^n \bold{V}(n,x)\bold{P} - \sum_{x=1}^{n+1}\bold{V}(n+1,x) \right).
	\end{equation}
	Finally, for the case that $x=0$, by conditioning on the first period of time, we note that for $n>1$
	\begin{equation*}
		\mathbb{P}(\tau^-_0 = n, J_n) = \bold{\Lambda}(0)\mathbb{P}_{1}(\tau^-_0 = n-1, J_{n-1}) 
	\end{equation*}
	and the result follows immediately by substitution of Eq.\,\eqref{eq:dist1} in the above. 
\end{proof}
\begin{Remark}
It is worth pointing out that although this section is concerned with the case of zero initial surplus ($x=0$), within the proof of the above theorem we have also derived an expression for the distribution of the time to ruin when $x=1$ in Eq.\eqref{eq:dist1}. 
\end{Remark}

\section{Finite-Time Ruin - $x\geq0$}

For $x \geq 0$, it is clear that the ballot-type theorem alone is not sufficient as this only contains information regarding the excursion above the initial point ($x$) and some additional arguments are required. One approach, considered by \cite{lefevre2008finite}, is to consider the corresponding time-reversed process for which the problem can then be expressed in terms of an upper hitting time. In the scalar case, the time-reversed process has identical distribution to that of the original process and the upper hitting time quantity (Kemperman's formula) is equivalent to the Ballot Theorem (see \cite{lefevre2008finite} for details). However, in the Markov-modulated setting, the time-reversed process does not have the same distribution as the original process. Moreover, the required upper hitting time quantity is replaced by a joint distribution with respect to the state of the Markov chain at the stopping time. These additional complexities mean that we cannot apply the same methodology as \cite{lefevre2008finite} and instead, consider a Lagrangian inversion approach by taking advantage of the form of Eq.\,\eqref{eq:upper}.

We now come to the main result of this paper, given in the following theorem, which derives a Seal-type formula for the compound Markov binomial model or upward skip-free Markov-modulated random walk. 

\begin{Theorem}[Seal-Type Formula] \label{thm:seal}
	For $n \geq 1$, we have 
	\begin{equation}\label{mainidentballot}
		\mathbb{P}_{\boldsymbol{\pi}, x}(\tau_0^-\geq
		n+1)=\boldsymbol{\pi}\left(\sum_{i=0}^{x+n-1}\bold{\Lambda}^{*n}(i) - \sum_{j=x+1}^{x+n-1} \sum_{\nu=j}^{x+n-1}  \bold{\Lambda}^{*j-x}(j)\bold{V}(n+x-j,n+x-\nu)\right)\vec{\boldsymbol{e}} .
	\end{equation}
\end{Theorem}
\begin{proof}

Let us consider the reversed time processes $(\widetilde{X}, \widetilde{J}\,)$. Then, it follows that for $0 < m \leq  x+n$ and $x \in  \mathbb{N} $, we have
\begin{eqnarray*}
	\mathbb{P}_{\boldsymbol{\pi}}\left( X_k > 0 \,\, \forall 0 < k < n, X_n = m \big| X_0 = x\right) = \mathbb{P}_{\boldsymbol{\pi}}\left( \widetilde{\tau}^{\,+}_m \geq n, \widetilde{X}_n = m-x \right),
\end{eqnarray*}
where $\widetilde{\tau}^+_m$ was defined in Section \ref{sec:rev}.
Moreover, for $m,x \in \mathbb{N}$, we also have 
\begin{eqnarray*}
	\mathbb{P}_{\boldsymbol{\pi}}\left(\widetilde{X}_n  = m-x   \right)  &=& 	\mathbb{P}_{\boldsymbol{\pi}}\left(\widetilde{\tau}^{\,+}_m \geq n, \widetilde{X}_n  = m-x   \right) \\
	&&+ \sum_{k=m}^{n-1}\sum_{j=1}^N \mathbb{P}_{\boldsymbol{\pi}}\left(\widetilde{\tau}^{\,+}_m = k, \widetilde{J}_{\widetilde{\tau}^{\,+}_m} = j  \right) \\
	&&\hspace{5mm}\times \mathbb{P}\left(\widetilde{X}_{n-k}  =-x |\widetilde{J}_{0} = j	\right),
\end{eqnarray*}

after applying the Markov additive property in the last term, and thus

\begin{eqnarray} \label{eq:genx1}
\mathbb{P}_{\boldsymbol{\pi}}\left(\widetilde{\tau}^{\,+}_m \geq n, \widetilde{X}_n  = m-x   \right) &=& 	\mathbb{P}_{\boldsymbol{\pi}}\left(\widetilde{X}_n  = m-x   \right) \notag  \\
	&&- \sum_{k=m}^{n-1}\sum_{j=1}^N \mathbb{P}_{\boldsymbol{\pi}}\left(\widetilde{\tau}^{\,+}_m = k, \widetilde{J}_{\widetilde{\tau}^{\,+}_m} = j  \right)  \mathbb{P}\left(\widetilde{X}_{n-k}  =-x |\widetilde{J}_{0} = j	\right) \notag  \\
	&=& \mathbb{P}_{\boldsymbol{\pi}}\left(\widetilde{S}_n  = n+x-m \right) \notag \\ &&-  \sum_{k=m}^{n-1}\sum_{j=1}^N \mathbb{P}_{\boldsymbol{\pi}}\left(\widetilde{\tau}^{\,+}_m = k, \widetilde{J}_{\widetilde{\tau}^{\,+}_m} = j  \right) \mathbb{P}\left(\widetilde{S}_{n-k}  = n+x-k \big | \widetilde{J}_0 = j\right) \notag \\
	&=& \boldsymbol{\pi}\widetilde{\bold{\Lambda}}^{*n}(n+x-m)\vec{\boldsymbol{e}} -\boldsymbol{\pi}\sum_{k=m}^{n-1} \Delta_{\boldsymbol{\pi}}^{-1} \bold{V}^{\top}(k,m)\Delta_{\boldsymbol{\pi}}\widetilde{\bold{\Lambda}}^{*n-k}(n+x-k)\vec{\boldsymbol{e}}\notag \\ 
	&=& \boldsymbol{\pi}\bold{\Lambda}^{*n}(n+x-m)\vec{\boldsymbol{e}} -\boldsymbol{\pi}\sum_{k=m}^{n-1} \bold{\Lambda}^{*n-k}(n+x-k)\bold{V}(k,m)\vec{\boldsymbol{e}},\notag \\
\end{eqnarray}
where, in the last equality, we have used Eq.\,\eqref{eq:TRP} and the fact that the individual terms are scalar, i.e., equal to their own transpose. 

Hence, the finite-time survival (ruin) probability can be expressed as 
	\begin{eqnarray*} \label{eq:ruin1}
		\mathbb{P}_{\boldsymbol{\pi}, x}(\tau_0^-\geq
		n+1)&=&\sum_{m=1}^{x+n} 
		\mathbb{P}_{\boldsymbol{\pi}}\left( X_k > 0 \,\, \forall 0< k < n, X_n = m \big| X_0 = x\right) \notag \\
		&=&\sum_{m=1}^{x+n} \mathbb{P}_{\boldsymbol{\pi}}\left( \widetilde{\tau}^{\,+}_m \geq n, \widetilde{X}_n = m-x \right) \\
		&=&
		\sum_{m=1}^{x+n} \boldsymbol{\pi}\bold{\Lambda}^{*n}(n+x-m)\vec{\boldsymbol{e}} - \boldsymbol{\pi} \sum_{m=1}^{x+n}\sum_{k=m}^{n-1} \bold{\Lambda}^{*n-k}(n+x-k)\bold{V}(k,m) \vec{\boldsymbol{e}} \\
		&=&
		\sum_{i=0}^{x+n-1} \boldsymbol{\pi}\bold{\Lambda}^{*n}(i)\vec{\boldsymbol{e}} - \boldsymbol{\pi} \sum_{j=x+1}^{x+n-1}\sum_{\nu=j}^{x+n-1} \bold{\Lambda}^{*j-x}(j)\bold{V}(n+x-j,n+x-\nu) \vec{\boldsymbol{e}}. \\
	\end{eqnarray*}
	The result follows immediately by factorising both $\boldsymbol{\pi}$ and $\vec{\boldsymbol{e}}$ on the left and right respectively. 
\end{proof}

\begin{Remark}
Note that the above result includes the case $x=0$. Hence, by observing that 
\begin{equation*}
	\mathbb{P}_{\boldsymbol{\pi}}(\widetilde{\tau}^+_m \geqslant n, \widetilde{X}_n = m) = \mathbb{P}_{\boldsymbol{\pi}}(\widetilde{\tau}^+_m = n),
\end{equation*}
 comparing Corollary \ref{cor:rev} with Eq.\,\eqref{eq:genx1}, we find 
\begin{equation*}
\boldsymbol{\pi} \bold{V}(n,m)\vec{\boldsymbol{e}}	=\boldsymbol{\pi}\bold{\Lambda}^{*n}(n-m)\vec{\boldsymbol{e}} - \boldsymbol{\pi}\sum_{k=m}^{n-1}  \bold{\Lambda}^{*n-k}(n-k)\bold{V}(k,m)\vec{\boldsymbol{e}}
\end{equation*}
and thus 
\begin{equation*}
	\boldsymbol{\pi}\left(\bold{\Lambda}^{*n}(n-m) - \sum_{k=m}^{n}  \bold{\Lambda}^{*n-k}(n-k)\bold{V}(k,m)\right)\vec{\boldsymbol{e}} = 0.
	\end{equation*}
In the case of no modulation (scalar), i.e.\,$N=1$, it can be shown that $\bold{V}(k,m) = (m/k)\lambda^{*k}(k-m)$, Theorem \ref{thm:seal} reduces to Proposition 2.4 of \cite{lefevre2008finite} and the above identity reduces to Eq.\,(2.16) of the same paper.
\end{Remark}

\begin{Remark}
	The general line of logic contained within the proof above, i.e.\,expressing the survival probability as an upper hitting time for the time-reversed process, would also hold in the non-stationary case. The main issue arises with obtaining an expression for this hitting time since, in the non-stationary case, the time-reversed process $\{\widetilde{J}_n\}_{n\in \mathbb{N}}$ is a non-homogenous Markov chain and thus $\{(\widetilde{X}_k, \widetilde{J}_k)\}_{k \in \mathbb{N}}$ no longer satisfies the Markov additive property.
\end{Remark}

\section*{Acknowledgements}
Z. Palmowski acknowledges that the research is partially supported by the Polish National Science Centre Grant No. 2021/41/B/HS4/00599.

\bibliography{Bibliography(Ballot)}

\section*{Appendix}
\subsection*{A1. Proof of Proposition \ref{Prop:Increm}}

For $1 \leq m \leq n$ and $r \leq n-m$, consider an arbitrary set of ordered integers $0=i_0 < i_1 < \cdots < i_r \leq n-m$. Then, under the initial stationary distribution $\boldsymbol{\pi}$, for any $k_0, \ldots, k_{r} \in \mathbb{N}$, we have 
\begin{eqnarray} \label{eq:App1}
&&\mathbb{P}_{\boldsymbol{\pi}}\left( C_{m+i_0} = k_0, \ldots, C_{m+i_r}=k_r\right) \notag \\
&&\hspace{5mm}=
\sum_{i, j_{1} \in E} \mathbb{P}\left( C_{m+i_0} = k_0, \ldots, C_{m+i_r}=k_r | J_{m+i_0}=j_{1}, J_{m+i_0-1}=i\right) \notag \\
&&\hspace{75mm}\times \mathbb{P}_{\boldsymbol{\pi}}(J_{m+i_0}=j_1, J_{m+i_0-1}=i) \notag \\
&&\hspace{5mm}=\sum_{i, j_{1} \in E} \mathbb{P}\left( C_{m+i_0} = k_0 | J_{m+i_0}=j_{1}, J_{m+i_0-1}=i\right) \notag \\
&&\hspace{20mm}\times \mathbb{P}\left( C_{m+i_1} = k_1, \ldots, C_{m+i_r}=k_r | J_{m+i_0}=j_{1}, J_{m+i_0-1}=i\right) \notag   \\
&&\hspace{75mm}\times \mathbb{P}_{\boldsymbol{\pi}}(J_{m+i_0}=j_1, J_{m+i_0-1}=i) \notag \\
&&\hspace{5mm}=\sum_{i, j_{1} \in E} \pi_i p_{i,j_1}(k_0) \mathbb{P}\left( C_{m+i_1} = k_1, \ldots, C_{m+i_r}=k_r | J_{m+i_0}=j_{1} \right),
\end{eqnarray}
where, in the second equality, we have used the conditional independence of the random variables $\{C_k\}_{k \in \mathbb{N}^+}$. Following a similar line of logic, by conditioning on the state of the Markov chain at times $m+i_1-1$ and $m+i_1$, we have 
\begin{eqnarray}
	&&\mathbb{P}\left( C_{m+i_1} = k_1, \ldots, C_{m+i_r}=k_r | J_{m+i_0}=j_{1} \right) \notag \\
	&&\hspace{5mm}= \sum_{j_2, j_3 \in E} 	\mathbb{P}\left( C_{m+i_1} = k_1, \ldots, C_{m+i_r}=k_r | J_{m+i_1}=j_3, J_{m+i_1-1}=j_2 \right) \notag  \\
	&&\hspace{75mm}\times \mathbb{P}\left( J_{m+i_1}=j_3, J_{m+i_1-1}=j_2|J_{m+i_0}=j_1\right) \notag \\
	&&\hspace{5mm}= \sum_{j_2, j_3 \in E} \mathbb{P}\left( C_{m+i_1} = k_1| J_{m+i_1}=j_3, J_{m+i_1-1}=j_2 \right) \notag \\
	&&\hspace{25mm} \times \mathbb{P}\left( C_{m+i_2} = k_2, \ldots, C_{m+i_r}=k_r | J_{m+i_1}=j_3\right) \notag \\
	&&\hspace{75mm}\times \mathbb{P}\left( J_{m+i_1}=j_3, J_{m+i_1-1}=j_2|J_{m+i_0}=j_1\right) \notag \\ 
		&&\hspace{5mm}= \sum_{j_2, j_3 \in E} p_{j_1, j_2}^{(i_1-1-i_0)} p_{j_2, j_3}(k_1) \mathbb{P}\left( C_{m+i_2} = k_2, \ldots, C_{m+i_r}=k_r | J_{m+i_1}=j_3\right),
\end{eqnarray}
where $p_{i,j}^{(n)}$ is the n-step transition probability from state $i$ to $j \in E$. Substituting this back in to Eq.\,\ref{eq:App1}, we have 
\begin{eqnarray*}
\mathbb{P}_{\boldsymbol{\pi}}\left( C_{m+i_0} = k_0, \ldots, C_{m+i_r}=k_r\right) &=& \sum_{i,j_1, j_2, j_3 \in E} \pi_i p_{i,j_1}(k_0) p_{j_1, j_2}^{(i_1-1-i_0)} p_{j_2, j_3}(k_1) \\
&&\hspace{10mm}\times \mathbb{P}\left( C_{m+i_2} = k_2, \ldots, C_{m+i_r}=k_r | J_{m+i_1}=j_3\right).
\end{eqnarray*}
Finally, it is not hard to see that repeating this process yields 
\begin{eqnarray*}
	\mathbb{P}_{\boldsymbol{\pi}}\left( C_{m+i_0} = k_0, \ldots, C_{m+i_r}=k_r\right) &=& \sum_{i,j_1, \cdots, j_{2r}, \ell \in E} \pi_i p_{i,j_1}(k_0) p_{j_1, j_2}^{(i_1-1-i_0)} p_{j_2, j_3}(k_1) \cdots p_{j_{2r-1},j_{2r}}^{(i_r-1-i_{r-1})}p_{j_{2r},\ell}(k_r)
\end{eqnarray*} 
which is independent of $m$ and thus, holds for any $1\leq m \leq n$. It is also worth noting that a similar argument holds for general jump sizes $\{C_k\}_{k \in \mathbb{N}^+}$, not just those with integer support. This completes the proof. 

\subsection*{A2. $\bold{Q}(n,a)$ for $N=1$}

First note that when $N=1$, $\bold{Q}(n,a)$ defined in Eq.\eqref{eq:G}, reduces to 
\begin{equation*}
	\bold{Q}(n,a) = \frac{1}{n!} b(n, a),
\end{equation*}
where 
\begin{eqnarray*}
	b(n,a) &=& \frac{\partial^n}{\partial z^n} \left[z^a \widehat{\lambda}(z)^n \left(1-\frac{z\widehat{\lambda}'(z)}{\widehat{\lambda}(z)}\right) \right]_{z=0} \\
	&=& \frac{\partial^n}{\partial z^n} \left[z^a \left(\widehat{\lambda}(z)^n -z\widehat{\lambda}(z)^{n-1}\widehat{\lambda}'(z)\right) \right]_{z=0} \\
	&=& \frac{\partial^n}{\partial z^n} \left[z^a \left(\widehat{\lambda}(z)^n -n^{-1}z \frac{\partial}{\partial z}\left(\widehat{\lambda}(z)^n\right)\right) \right]_{z=0}.
\end{eqnarray*}
Applying general Leibniz (product) rule gives
\begin{eqnarray*}
	b(n,a) &=& \left[\sum_{i=0}^n \left(\begin{array}{c} n \\ i \end{array} \right) \frac{\partial^i }{\partial z^i } (z^a) \times \frac{\partial^{n-i}}{\partial z^{n-i}} \left(\widehat{\lambda}(z)^n -n^{-1}z \frac{\partial}{\partial z}\left(\widehat{\lambda}(z)^n\right) \right) \right]_{z=0} \\
	&=& \left[\sum_{i=0}^n \left(\begin{array}{c} n \\ i \end{array} \right) \frac{\partial^i }{\partial z^i } (z^a) \left[ \frac{\partial^{n-i}}{\partial z^{n-i}} \left(\widehat{\lambda}(z)^n\right) -n^{-1}(n-i) \frac{\partial^{n-i}}{\partial z^{n-i}}\left(\widehat{\lambda}(z)^n\right) \right) \right]_{z=0} \\
	&=& \left(\begin{array}{c} n \\ a \end{array} \right) a!  \left[ \frac{\partial^{n-a}}{\partial z^{n-a}} \left(\widehat{\lambda}(z)^n\right) -n^{-1}(n-a) \frac{\partial^{n-a}}{\partial z^{n-a}}\left(\widehat{\lambda}(z)^n\right) \right]_{z=0} \\
	&=&  \frac{a}{n}\frac{n!}{(n-a)!} \left[\frac{\partial^{n-a}}{\partial z^{n-a}} \left(\widehat{\lambda}(z)^n\right)\right]_{z=0} \\
	&=& \frac{a}{n}\frac{n!}{(n-a)!} (n-a)! \lambda ^{*n}(n-a) \\
	&=& \frac{a}{n}\lambda^{*n}(n-a),
\end{eqnarray*}
where $\lambda^{*n}(\cdot)$ denotes the $n$-th fold convolution of $\lambda(\cdot)$.

\end{document}